\documentclass[12pt]{article}

\usepackage[utf8]{inputenc}
\usepackage[english]{babel}
\usepackage[backend=biber,maxnames=5,style=lncs]{biblatex}
\usepackage{url}
\usepackage{hyperref}
\RequirePackage{doi}
\usepackage{todonotes}
\usepackage{appendix}
\usepackage[cmex10]{amsmath}
\usepackage{amssymb}
\usepackage{amsthm}
\usepackage{ifthen}
\usepackage{pgf}
\usepackage{tikz}
\usepackage{epsfig}

\title{Cyclic Relative Difference Sets and Circulant Weighing Matrices}

\begin{document}

\author{Daniel M. Gordon
  \thanks{D.M. Gordon is with the IDA Center for Communications
    Research-La Jolla, 4320 Westerra Court, San Diego, CA 92121, USA (email: gordon@ccr-lajolla.org)}
}
%\orcid{0000-0003-4373-3637}

\newcommand{\dnd}{\;\!\mid\mskip -9mu \not \;\;}
\def\Fnum{{\mathbb{F}}} % a field F
\def\Qnum{{\mathbb{Q}}} % a field F
\def\Nnum{{\mathbb{N}}} % a field F
\def\Cnum{{\mathbb{C}}} % a field F
\def\Znum{{\mathbb{Z}}} % a field F
\def\Rnum{{\mathbb{R}}} % a field R
\newcommand{\br}[1]{\langle {#1} \rangle }
\def\Mdot{{\dot{M}}}
\def\cO{{\cal O}}

\newtheorem{question}{Question} % numbered vs unnumbered?  2004 aug 15
\newtheorem{idea}{Idea} % numbered vs unnumbered?  2004 aug 15

\newtheorem{theorem}{Theorem}
\newtheorem{lemma}[theorem]{Lemma}
\newtheorem{conjecture}[theorem]{Conjecture}
\newtheorem{remark}{Remark}
\newtheorem{claim}{Claim}
\newtheorem{corollary}[theorem]{Corollary}
\newtheorem{proposition}[theorem]{Proposition}
\newtheorem{definition}[theorem]{Definition}

\maketitle

\abstract{
An $(m,n,k,\lambda)$-relative difference set is a lifting of an
$(m,k,n\lambda)$-difference set.  Lam gave a table of cyclic relative
difference sets with $k \leq 50$ in 1977, all of which were liftings
of $( \frac{q^d-1}{q-1},q^{d-1},q^{d-2}(q-1))$-difference sets, the
parameters of complements of classical Singer difference sets.  Pott
found all liftings of these difference sets with $n$ odd and $k \leq
64$ in 1995.  No other nontrivial difference sets are known with
liftings to relative difference sets, and Pott ended his survey on
relative difference sets asking whether there are any others.

In this paper we extend these searches, and apply the results to the
existence of circulant weighing matrices.
}

%\listoftodos

\section{Introduction}

A $(v,k,\lambda)$-difference set is a subset
$D=\{d_1,d_2,\ldots,d_k\}$ of a group $G$ such that every nonzero
element of $G$ is represented exactly $\lambda$ times as a difference
of two elements of $D$.
We will identify $D$ with its representation $D=\sum_{i=1}^k d_i$ in the group
ring $\Znum[G]$, which  satisfies
\begin{equation}\label{eq:gpring}
D D^{-1} = k-\lambda + \lambda G,
\end{equation}

There is a vast literature on difference
sets; see, for example \cite{bjl}.
An important class of difference sets are {\em Singer difference
  sets}, which are constructed from projective geometries
and have parameters
\begin{equation}\label{eq:singer}
  \left(
  \frac{q^{d}-1}{q-1},\frac{q^{d-1}-1}{q-1},\frac{q^{d-2}-1}{q-1} \right),
\end{equation}
for $q$ a prime power.

The complement of a $(v,k,\lambda)$-difference set is a
$(v,v-k,v-2k+\lambda)$-difference set.  The complement of a Singer
difference set has parameters
\begin{equation}\label{eq:singercomp}
\left( \frac{q^d-1}{q-1},q^{d-1},q^{d-2}(q-1) \right).
\end{equation}

%Relative difference sets are a generalization of difference sets.
An $(m,n,k,\lambda)$-relative difference set (RDS) $R$ in a group $G$
relative to a normal subgroup $N$ is a $k$-subset of $G$ such that the
differences of distinct elements of $R$ contain every element of
$G \backslash N$ exactly $\lambda$ times, and no non-identity element
of $N$.  Here $G$ has order $mn$, and the forbidden subgroup $N$ has
order $n$.  An RDS is called abelian if $G$ is abelian, and cyclic if
$G$ is cyclic.

The group ring equation for an RDS is:
\begin{equation}\label{eq:rds_gpring}
R R^{-1} = k  + \lambda (G-N),
\end{equation}

%% Counting the differences two ways, an easy necessary condition is
%% $$
%% k(k-1) = \lambda n (m-1).
%% $$

An $(m,1,k,\lambda)$-relative difference set is a difference set.
See \cite{pott} for background on relative difference
sets.

An RDS is {\em splitting} if $G$ is isomorphic to the direct product of $N$
and $G/N$.  For cyclic groups, that means that $m$ and $n$ are
relatively prime.

%% The following theorem shows that a relative difference set
%% is a lifting of a difference set \cite{elliott1966illinois}:

Elliott and Butson \cite{elliott1966illinois} first noted that
the projection of a relative difference set to a subgroup yields
another relative difference set:

\begin{theorem}\label{thm:lift}
  Let $R$ be an $(m,n,k,\lambda)$-RDS in $G$.  If $U$ is a normal
  subgroup of $G$ of order $u$ contained in $N$, then there exists an
  $(m,n/u,k,\lambda u)$-RDS in $G/U$ relative to $N/U$.
  In particular, $G/N$ contains an
  $(m,k,\lambda n)$-difference set.
\end{theorem}

In this case $R$ is called a {\em lifting} or {\em extension} of the
difference set.  
For example, $\{0,3,5,6\}$ is a $(7,4,2)$ difference set (the complement of the
$(7,3,1)$ projective plane of order two), and it may be lifted to
$\{0,3,5,13\}$, a $(7,2,4,1)$-RDS in
$\Znum_{14}$
relative to the subgroup $\{0,7\}$.

There are relative difference sets which are liftings of trivial
$(m,m,m)$-difference sets (these are called {\em semiregular}), and of
trivial $(m+1,m,m-1)$-difference sets.  While both of these cases have
interesting examples and open existence questions, in this paper we
will be concerned with cyclic relative difference sets which are lifts
of nontrivial difference sets.  The only ones known have the
parameters
of complements of classical Singer difference sets
\cite{adlm}:

\begin{theorem}\label{thm:singer}
For $q$ a prime power,
a cyclic   $\left( \frac{q^d-1}{q-1},n,q^{d-1},q^{d-2}(q-1)/n
\right)$-RDS exists if and only if
%% $n$ is a divisor of $q-1$ when $q$
%% is odd or $d$ is even, and if and only if $n$ is a divisor of $2(q-1)$
%% if $q$ is even and $d$ is odd.
$$
  \begin{array}{ll}
  n\, |\, 2(q-1), & {\rm if} \  q\ {\rm even\ and}\ d\ {\rm odd},\\
  n\,|\, (q-1), & {\rm else}.
  \end{array}
$$
\end{theorem}

Arasu, Jungnickel, Ma and Pott \cite{ajmp} show that many other difference
sets cannot have liftings with $n=2$, including:
\begin{itemize}
\item Paley difference sets
\item Twin prime power difference sets and their complements
\item difference sets with the parameters of classical Singer difference sets
\end{itemize}

They conjecture that only difference sets with parameters (\ref{eq:singercomp})
%of the complements of Singer difference sets
with $d$ odd
have liftings to relative difference sets with
$n=2$.
Pott \cite{pott} asked
whether other difference sets have liftings to
relative difference sets with {\em any} $n$.

Lam \cite{lam} gave a list of cyclic relative difference sets with $k \leq 50$.  Pott \cite{pott}
extended this (for Singer parameters) to $k \leq 64 $ for $n$ odd, and
also found all the $(121,2,81,27)$ liftings of the four cyclic
$(121,81,54)$-difference sets.  In this paper we will give the results
of further searches for lifts of difference sets with $k \leq 256$,
which found no liftings of any other
%types
parameters
of difference sets.

A {\em weighing matrix} $M = M(n,k)$ is an $n \times n$ matrix with
entries in $\{-1,0,1\}$ such that $M M^T = kI_n$, where $M^T$ is the
transpose of $M$ and $I_n$ is the $n \times n$ identity matrix.  If
$M$ is cyclically symmetric (every row is the cyclic shift of the row
above), then $M$ is a {\em circulant weighing matrix}.  This is
equivalent to an element $W$
of the  group ring of $\Znum_n$ with each coefficient
in $\{-1,0,1\}$, satisfying
\begin{equation}\label{eq:cw}
  W W^{-1} = k,
\end{equation}
which is the same as (\ref{eq:gpring}) with $\lambda=0$.

In \cite{agz} it is shown that most known circulant weighing
matrices may be constructed either with a product construction using
two smaller matrices, or from a relative difference set.  In
Section~\ref{sec:cw} we will use results of RDS searches to
give new existence results for circulant weighing matrices.

%% The circulant weighing matrices and relative difference sets
%% found in
%% this work are available at \cite{ljcwr} and \cite{ljrdsr}, respectively.

\section{Nonexistence Theorems for Relative Difference Sets}\label{sec:nonexist}

There are a number of nonexistence results that can quickly eliminate
some parameters, many given by Lam in \cite{lam}.  That reference is not
widely available; we will not reproduce the proofs here, but note that
they are straightforward generalizations of results for difference
sets, with proofs given by Baumert in \cite{baumert}.
Theorem 3.1 of \cite{lam} is:

\begin{theorem}\label{thm:lam3.1}
  If a cyclic $(q^2+q+1,q^2,q^2-q)$-difference set (the complement of a
  $(2,q)$ Singer plane) has a lifting with $n = p^r n_1$, where $p$ is
  prime, $q = p^s$, $r \geq 1$, then
  $$
  n_1^e \equiv 1 \pmod {p^r},
  $$
  where $e = 3s/\gcd(3s,r)$.
\end{theorem}

Let $p$ be a prime, and  let $w = p^e w_1$, with $\gcd(p,w_1)=1$.
We say that $p$ is {\em   self-conjugate modulo $w$} if there is an integer $f>0$ such that
$p^f \equiv -1 \pmod {w_1}$.  A composite $u$ is self-conjugate mod
$w$ if all of its prime divisors are.

Theorem 3.2 of \cite{lam} is:

\begin{theorem}\label{thm:lam3.2}
  Let $R$ be a cyclic $(m,n,k,\lambda)$-RDS, and $w>1$ be a divisor of $mn$,
  with $w \dnd m$.
  If there exists $u>1$ self-conjugate modulo $w$
  satisfying $u^2|k$, then
  $$
  u \leq 2^{s-1} (mn/w),
  $$
  where $s$ is the number of distinct prime factors of $w$.
\end{theorem}

Theorem 3.5 of \cite{lam} is:

\begin{theorem}\label{thm:lam3.5}
  Let $R$ be a cyclic $(m,n,k,\lambda)$-RDS, and let $q$ be a prime divisor of
  $mn$, where
  $q^e \parallel mn$, $q^e \dnd m$, and $q \equiv 3 \pmod 4$.
  If every prime $p|k$ satisfies one of the conditions:
  \begin{enumerate}
  \item the order of $p$ mod $q$ is even,
  \item the order of $p$ mod $q$ is $q^{e-1}(q-1)/2$, or
  \item $p=q$,
  \end{enumerate}
  then the Diophantine equation
  $$
  4k = x^2 + q y^2
  $$
  has a solution in integers $x$ and $y$.
\end{theorem}

Finally, we give Result 2.5 of Pott \cite{pott1996survey}.  Here
the symbol $(a,b)_p$ is the Hilbert symbol:
$$
(a,b)_p = \left\{
  \begin{array}{ll}
  1 & {\rm if\ } a x^2 + b y^2 \equiv 1 \pmod {p^r} {\rm \ has \ a
    \ rational\ solution \ for \ every}\  r,\\
  -1 & {\rm else}.
  \end{array}
  \right.
$$

\begin{theorem}\label{thm:pott2.5}
  If an abelian $(m,n,k,\lambda)$-RDS exists, then the following holds:
  \begin{enumerate}
  \item if $m$ is even, then $k-n\lambda$ is a square.  If $m \equiv 2
    \pmod 4$, and $n$ is even, then $k$ is the sum of two squares.
  \item If $m$ is odd and $n$ is even, then $k$ is a square and
    $$
    (k-n\lambda,(-1)^{(n-1)/2} n \lambda )_p = 1
    $$
    for all odd primes $p$.
  \item If both $m$ and $n$ are odd, then
    $$
    (k,(-1)^{(n-1)/2} n )_p \cdot (k-n\lambda,(-1)^{(m-1)/2} n \lambda )_p = 1
    $$
    for all odd primes $p$.
  \end{enumerate}
\end{theorem}

%% Lemma 2.6 of \cite{pott} states:
%% \begin{lemma}
%%   Let $R$ be a $(m,n,k,\lambda)$-RDS in $G$ relative to $N$.  If $U$
%%   is a normal subgroup of $G$ contained in $N$, then there exists a
%%   $(m,n/u,k,\lambda u)$-RDS in G/U relative to $N/U$.
%% \end{lemma}

%% This means that any RDS is the lifting of a difference set, and when
%% we look at the existence of RDS that are lifts of a given
%% $(v,k,\lambda)$-difference set, we need only list the ``maximal''
%% values of $n$, i.e. the largest $n'$ for which a
%% $(v,n,k,\lambda/n)$-RDS exists for all divisors of $n'$.

Using these theorems, we will show in the next section
that many difference sets do not
have liftings to relative difference sets.

\section{Multipliers}\label{sec:mult}

One of the main tools for studying difference sets  is {\em
  multipliers}.
%% %% For a subset $S$ of $G$, let
%% %% $$
%% %% S = \sum_{g \in S} g
%% %% $$
%% %% be the corresponding group ring element.
%% For any group ring element $A = \sum_{g \in G} a_g g$, define
%% $$
%% A^{(t)} = \sum_{g \in G} a_g g^t.
%% $$
Let $\{d_1, d_2,\ldots,d_k\}$ be a
$(v,k,\lambda)$-difference set in an abelian group $G$.
An integer $t$
%relatively prime to $v$
is called a
{\em multiplier} if multiplying the elements of $D$ by $t$
result in a translate of $D$:
%% $$
%% D^{(t)} = \sum_{i=1}^k t d_i = \sum_{i=1}^k (d_i+s)  = D+s
%% $$
$$
D^{(t)} = \{ t d_1, \ldots td_k\} = \{d_1+s,\ldots d_k+s\}  = D+s
$$
for some $s \in G$.  It is well known that some translate of
a difference set $D$ is fixed by all of its multipliers.
There are numerous theorems giving conditions for an integer $t$ to be
a multiplier; see \cite{gs}.

%Theorem 2.1 of \cite{lam}  
Theorem 1.3.5 of \cite{pott}
gives a condition for a multiplier of a
difference set to extend to a relative difference set:

\begin{theorem}\label{thm:lammult}
%(Lam's version)
  Let $R$ be an $(m,n,k,\lambda)$-RDS in an abelian group $G$ of
  exponent $v^*$ relative to $N$.  Let $t$ be an integer relatively
  prime to $v=mn$
which is a multiplier of the underlying
  $(m,k,n\lambda)$-difference set.
Let $k_1$ be a divisor of $k$, with prime factorization
  $k_1 = p_1^{e_1} \cdots p_r^{e_r}$, and $k_2 = k_1/\gcd(v,k_1)$.
  For each $p_i$, define
  $$
  q_i = \left\{
  \begin{array}{ll}
  p_i & {\rm if}\  p_i {\rm \ does \ not \  divide}\  v \\
  l_i & {\rm if}\  v^* = p_i^r u_i,\  \gcd(p_i,u_i)=1, \ {\rm
    where}\  l_i \ 
  {\rm is \ an \ integer \  such \ that \ }\\
  & \gcd(l_i,p_i)=1 \ {\rm and}\ \  l_i \equiv p_i^f \pmod   u_i.
  \end{array}
  \right.
  $$
  If $k_2 > \lambda$, and for each $i$ there exists an integer $f_i$
  such that $ q_i^{f_i} \equiv t \pmod {v^*}$,
 then $t$ is a multiplier of $R$.
\end{theorem}

%% Theorem 1.3.5 of \cite{pott} gives this multiplier theorem for relative
%% difference sets:

%% \begin{theorem} (Pott's version)
%%   Let $R$ be an $(m,n,k,\lambda)$-RDS in an abelian group $G$ of
%%   exponent $v^*$ relative to $N$.  Let $t$ be an integer relatively prime to $v=mn$,
%%   and let $k_1$ be a divisor of $k$, with prime factorization
%%   $k_1 = p_1^{e_1} \cdots p_r^{e_r}$, and $k_2 = k_1/\gcd(v,k_1)$.
%%   For each $p_i$, define
%%   $$
%%   q_i = \left\{
%%   \begin{array}{ll}
%%   p_i & {\rm if}\  p_i {\rm \ does \ not \  divide}\  v \\
%%   l_i & {\rm if}\  v^* = p_i^r u_i,\  \gcd(p_i,u_i)=1, \ {\rm where}\  l_i
%%   {\rm is \ an \ integer \  such \ that \ }\\
%%   & \gcd(l_i,p_i)=1 \ {\rm and}\ \  l_i \equiv p_i^f \pmod   u_i.
%%   \end{array}
%%   \right.
%%   $$
%%   Suppose for each $i$ there exists an integer $f_i$ and multiplier $s_i$
%%   such that $s_i q_i^{f_i} \equiv t \pmod {v^*}$.  If
%% $k_2 >  \lambda$, or if the equation $F F^{-1} = (k/k_2)^2$ in $\Znum[G]$ has
%%   only the trivial solution $F = (k/k_2)g$, then $t$ is a multiplier
%%   of $R$.
%% \end{theorem}

Theorem 1.3.8 of \cite{pott}
  %Pott says this was "essentially" in Elliott and Butson, but refers to BJL
gives a condition for a translate of $R$ to be fixed by one or all multipliers.

\begin{theorem}
  Let $R$ be an  abelian $(m,n,k,\lambda)$-RDS with
  multiplier $\tau$.  If $R$ not a lift of a
  $(m,m,m)$-difference set, then at least one translate of $R$ is
  fixed by $\tau$.  If $mn$ and $k$ are relatively prime, then at
  least one translate of $R$ is fixed by {\em all} multipliers.
\end{theorem}

For $w | mn$, let $\overline{D}$ be the natural map from
$\Znum_{mn}$ to $\Znum_w$, reducing
modulo $w$.
Define a {\em $w$-multiplier} to be an integer $t$
coprime to $w$ for which there an $s \in \Znum_w$ satisfying
$$
\overline{D}^{(t)} = \overline{D}+s.
$$
Theorem 3.4 of \cite{lam} gives more conditions on cyclic relative difference
sets with such multipliers:

\begin{theorem}\label{thm:lam3.4}
  If a cyclic $(m,n,k,\lambda)$-RDS exists, let $w|mn$ such that $w \dnd m$,
  and let $t$ be a $w$-multiplier.  Let
  $t^f \equiv -1 \pmod w$ for some integer $f$.  Then either $k$ is a
  square, or $k = k_1^2 q$ for some prime $q|w$.  In the latter case,
  \begin{enumerate}
  \item if $p \neq q$ is a prime, and $p^\alpha \parallel w$, then
    $p^\alpha | m$,
  \item if $q$ is odd and $q \dnd m$, then $q \equiv 1 \pmod 4$,
  \item if $q=2$, then $4|w$, and
  \item if $q \dnd m$, then $t$ is a quadratic residue modulo $q$.
  \end{enumerate}
\end{theorem}

For $w|mn$, let $b_i$ be the number of elements of an RDS equal to $i$ modulo
$w$.  The $b_i$ are called the {\em intersection numbers} (see Section VI.5
of \cite{bjl}), and equations involving them have been used to speed
searches for
difference sets
(see, for example, \cite{gaalgolomb}), relative difference sets
(\cite{pott}, Section 3.2), and circulant weighing matrices
\cite{agz}.  For relative difference sets, taking $w=n$ the equations become:

\begin{lemma}\label{lem:contracted}
For a cyclic $(m,n,k,\lambda)$-RDS $R$ with $d = \gcd(n,m)$,
\begin{align}
\sum_{i=0}^{n-1} b_i & =  k,\label{eq:lin}\\
\sum_{i=0}^{n-1} b_i^2 & = k + \lambda \cdot (m-d) ,\label{eq:quad}
\end{align}
where $b_i \leq m$.
\end{lemma}

\begin{proof}
The bound on the $b_i$ is obvious.
(\ref{eq:lin}) just says that $R$ has $k$ elements, and
  (\ref{eq:quad}) follows from reducing the group ring equation
(\ref{eq:rds_gpring}) modulo $n$ and evaluating the coefficient of $0$.
\end{proof}

These theorems give a way to investigate relative difference sets:
\begin{enumerate}
\item Start with an $(m,k,\lambda n)$ difference set, and find its set
  of multipliers $M_1$.
\item Check whether any of the theorems of the previous section
  exclude a lift to an $(m,n,k,\lambda)$-RDS.
\item Find multipliers  $M_2 = \{ t_1,t_2,\ldots,t_s\} \subset M_1 $ of $R$ using Theorem~\ref{thm:lammult}.
\item If $\gcd(mn,k)=1$, let $M$ be the subgroup of $G$ generated by
  $M_2$.  Otherwise, let $M$ be the subgroup generated by one $t_i$.
\item Search for a collection of orbits of $G/M$ which form an RDS.
\end{enumerate}

For example, consider the $(73,64,28)$-difference set.  Pott found the
unique $(73,7,64,4)$-relative difference set that it lifts to.  To find
lifts to a $(73,2,64,14)$-relative difference set, note that by the
First Multiplier Theorem, 2 is a multiplier
of the $(73,9,1)$-difference set and its complement, so the multiplier
group contains
the powers of two modulo 73: $\br{2}_9 = \{1,2,4,8,16,32,37,55,64\}$ (and in fact
that is the complete multiplier group $M_1$).  As in \cite{agz}, we will
use $\br{o}_s$ to denote the orbit of size $s$ generated by $o$.

The multiplier group $M$ of the RDS is $\br{75}_9$ (a lift of $M_1$ to $G=\Znum_{146}$).
The orbits of $G/M$ consist of $\br{0}_1$ and sixteen orbits of order
nine.  The $(73,64,28)$-difference set contains each orbit except
$\br{2}_9$, and each orbit is lifted to either $\br{0}_1$ or
$\br{1}_1$ in $\Znum_2$ (i.e. all elements in the orbit are unchanged,
or have $73$ added to them, respectively).

From Lemma~\ref{lem:contracted}, we have
$b_0 + b_1 = 64$, and $b_0^2 + b_1^2 = 2080$, with $0 \leq b_i \leq
73$.  There are two solutions: $(b_0, b_1) = (36,28)$ and $(28,36)$.
Without loss of generality we may take the first solution, and a
search reveals the unique RDS given in Table~\ref{tab:73}.

\begin{table*}
  \centering
\caption{Orbits in the $(73,2,64,28)$-RDS.  The rows correspond to
  orbits of $\Znum_{73}$, the columns to $\Znum_{2}$, and the
  entries in the table are the orbits of $(\Znum_{146})/M$ in the
  RDS.  The column sums are a set of $b_i$'s satisfying
  Lemma~\ref{lem:contracted}, and the row sums enforce it
  being a lifting of the $(73,64,28)$-DS.}
\label{tab:73}
\begin{tabular}{c||c|c||c}
 \multicolumn{3}{c}{\ \ \ $[2]$} \\
$[73]$ & $\br{ 0 }_{1}$  & $\br{ 1 }_{1}$ \\ \hline \hline
$\br{ 0 }_{1}$  &  & $\br{ 73 }_{1}$  & 1  \\ \hline
$\br{ 1 }_{9}$  &  &  & 0  \\ \hline
$\br{ 3 }_{9}$  & $\br{ 6 }_{9}$  &  & 9  \\ \hline
$\br{ 5 }_{9}$  & $\br{ 10 }_{9}$  &  & 9  \\ \hline
$\br{ 9 }_{9}$  & $\br{ 18 }_{9}$  &  & 9  \\ \hline
$\br{ 11 }_{9}$  &  & $\br{ 11 }_{9}$  & 9  \\ \hline
$\br{ 13 }_{9}$  &  & $\br{ 13 }_{9}$  & 9  \\ \hline
$\br{ 17 }_{9}$  & $\br{ 34 }_{9}$  &  & 9  \\ \hline
$\br{ 25 }_{9}$  &  & $\br{ 25 }_{9}$  & 9  \\ \hline
 &  36 &  28 & \end{tabular}
\end{table*}

The author maintains a website \cite{ljcr} of known abelian
difference sets for a wide range of parameters.  In particular, the
existence of cyclic difference sets for all parameters $(v,k,\lambda)$ with
$k <  105$ is known (the smallest open cases are $(1561,105,7)$,
$(2185,105,5)$, $(1111,111,11)$ and $(465,145,45)$).
A search was done for liftings of all known cyclic difference sets
with $k \leq 256$.  Note that this may not be a complete list;
for some parameters a difference set is known, but others might exist.

Tables~\ref{tab:new}, \ref{tab:new2},
\ref{tab:newbig} and \ref{tab:newbigger} give the results for
difference sets with different ranges of $k$ from $50$ to $256$.
Parameters eliminated by \cite{ajmp} (Paley, Singer or twin prime
power difference sets where $\lambda$ is a power of two) are omitted.
All parameters were settled except for $(364,121,40)$, $(255,127,63)$,
$(1464,133,12)$ and $(2380,183,14)$ (all Singer parameters),
for which none of the theorems
applied, and the exhaustive search was impractical.
The fact that no lifts were found for parameters other than
(\ref{eq:singercomp}) strengthens the evidence for the conjecture that
none exist. 

Lifts of complements of classical Singer parameters are given separately in
Table~\ref{tab:singer}, with the number of inequivalent RDS given when known.
Note that there are non-Singer complement difference sets with these
parameters that have lifts; all of the $(121,81,27)$ difference sets
(these lifts were found by Pott \cite{pott}), all of the $(364,243,162)$
difference sets, and one of the $(511,256,128)$ difference sets
(the GMW difference set \cite{gmw}).

%\todo{Discuss inequivalent RDS, relate to other papers} % In 2003 and 2005 some of these were constructed

\begin{table*}
  \centering
\caption{Possible liftings with $50 < k < 100$}
\label{tab:new}
\begin{tabular}{|c|c|c|c|c|}  \hline
$v$ & $k$ & $\lambda$ & type  &  nonexistence proof \\ \hline
103 & 51 & 25 & Paley & Thm.~\ref{thm:pott2.5} \\
103 & 52 & 26 & complement of Paley & $n=2$: Thm.~\ref{thm:pott2.5}, $n=13$: Thm.~\ref{thm:lam3.4} \\
107 & 53 & 26 & Paley & $n=2$: Thm.~\ref{thm:pott2.5}, $n=13$: Thm.~\ref{thm:lam3.4}\\
107 & 54 & 27 & complement of Paley & Thm.~\ref{thm:pott2.5} \\
400 & 57 & 8 & (3,7) Singer & \cite{ajmp} \\
127 & 63 & 31 & (6,2) Singer & Thm.~\ref{thm:lam3.2} \\
%127 & 64 & 32 & complement of (6,2) Singer & 2 (exhaustive search) \\
%73 & 64 & 56 & complement of (2,8) Singer & 14 (exhaustive search)\\
%85 & 64 & 48 & complement of (3,4) Singer & 3 (exhaustive search) \\
131 & 65 & 32 & Paley & \cite{ajmp}  and Thm.~\ref{thm:pott2.5}\\
131 & 66 & 33 & complement of Paley & Thm.~\ref{thm:pott2.5} \\
139 & 69 & 34 & Paley & Thm.~\ref{thm:pott2.5} \\
139 & 70 & 35 & complement of Paley & Thm.~\ref{thm:pott2.5}\\
143 & 71 & 35 & TPP(11) & Thm.~\ref{thm:lam3.4}\\
143 & 72 & 36 & complement of TPP(11) & Thm.~\ref{thm:pott2.5}  \\
585 & 73 & 9 & (3,8) Singer & search \\
151 & 75 & 37 & Paley & Thm.~\ref{thm:lam3.4} \\
101 & 76 & 57 & complement of Type B (Hall) & Thm.~\ref{thm:lam3.4} \\
151 & 76 & 38 & complement of Paley & $n=2$: Thm.~\ref{thm:pott2.5}, $n=19$: Thm.~\ref{thm:lam3.4} \\
109 & 81 & 60 & complement of Type B0 (Hall) & Thm.~\ref{thm:lam3.4}\\
%121 & 81 & 54 & complement of (4,3) Singer & 2 (no $n=3$ by search)\\
163 & 81 & 40 & Paley & $n=2$: Thm.~\ref{thm:lam3.2}, $n=5$: search \\
%91 & 81 & 72 & complement of (2,9) Singer & 2 (search) \\
163 & 82 & 41 & complement of Paley & Thm.~\ref{thm:lam3.4} \\
167 & 83 & 41 & Paley & Thm.~\ref{thm:lam3.4} \\
167 & 84 & 42 & complement of Paley & $n=2,7$: Thm.~\ref{thm:pott2.5}, $n=3$: Thm.~\ref{thm:lam3.4}  \\
341 & 85 & 21 & (4,4) Singer & Thm.~\ref{thm:pott2.5}\\
179 & 89 & 44 & Paley &  $n=2$: Thm.~\ref{thm:pott2.5}, $n=11$: Thm.~\ref{thm:lam3.4}  \\
179 & 90 & 45 & complement of Paley & Thm.~\ref{thm:pott2.5} \\
820 & 91 & 10 & (3,9) Singer & $n=2$: search, $n=5$: Thm.~\ref{thm:lam3.4}\\
191 & 95 & 47 & Paley & Thm.~\ref{thm:pott2.5} \\
191 & 96 & 48 & complement of Paley & Thm.~\ref{thm:pott2.5} \\
199 & 99 & 49 & Paley & Thm.~\ref{thm:lam3.2}\\ \hline
\end{tabular}
\end{table*}

\begin{table*}
  \centering
\caption{Possible liftings with $100 \leq k < 150$}
\label{tab:new2}
\begin{tabular}{|c|c|c|c|c|}  \hline
$v$ & $k$ & $\lambda$ & type  & nonexistence proof \\ \hline
133 & 100 & 75 & complement of Hall (1956)  & search \\  % for n=5, 123 gave nontrivial multiplier group
199 & 100 & 50 & complement of Paley &  search \\
211 & 105 & 52 & Paley &  Thm.~\ref{thm:pott2.5} \\
211 & 106 & 53 & complement of Paley &  Thm.~\ref{thm:lam3.4} \\
223 & 111 & 55 & Paley &  $n=5$: Thm.~\ref{thm:pott2.5}, $n=11$: Thm.~\ref{thm:lam3.4} \\
223 & 112 & 56 & complement of Paley & $n=2$: Thm.~\ref{thm:pott2.5}, $n=7$: Thm.~\ref{thm:lam3.4}  \\
227 & 113 & 56 & Paley &  $n=2$: Thm.~\ref{thm:pott2.5}, $n=7$: Thm.~\ref{thm:lam3.4}  \\
227 & 114 & 57 & complement of Paley &  Thm.~\ref{thm:pott2.5} \\
239 & 119 & 59 & Paley &  Thm.~\ref{thm:pott2.5} \\
239 & 120 & 60 & complement of Paley & Thm.~\ref{thm:pott2.5}  \\
%133 & 121 & 110 & complement of (2,11) Singer &  OPEN \\
364 & 121 & 40 & (5,3) Singer &  $n=2$: search, $n=5$: OPEN \\
%156 & 125 & 100 & complement of (3,5) Singer &  OPEN \\
251 & 125 & 62 & Paley &  $n=2$: Thm.~\ref{thm:pott2.5}, $n=31$: Thm.~\ref{thm:lam3.4}  \\
251 & 126 & 63 & complement of Paley &  $n=2$: Thm.~\ref{thm:pott2.5}, $n=7$: Thm.~\ref{thm:lam3.4}  \\
255 & 127 & 63 & (7,2) Singer & OPEN  \\
%255 & 128 & 64 & complement of (7,2) Singer & Thm.~\ref{thm:pott2.5}  \\
263 & 131 & 65 & Paley &  Thm.~\ref{thm:lam3.4} \\
263 & 132 & 66 & complement of Paley & $n=2,3$: Thm.~\ref{thm:pott2.5}, $n=11$: Thm.~\ref{thm:lam3.4}   \\
1464 & 133 & 12 & (3,11) Singer & OPEN  \\
271 & 135 & 67 & Paley & Thm.~\ref{thm:pott2.5}  \\
271 & 136 & 68 & complement of Paley &  $n=2$: Thm.~\ref{thm:pott2.5}, $n=17$: Thm.~\ref{thm:lam3.4}  \\
283 & 141 & 70 & Paley & Thm.~\ref{thm:pott2.5}  \\
283 & 142 & 71 & complement of Paley & Thm.~\ref{thm:lam3.5}   \\
197 & 148 & 111 & complement of Type B (Hall) & Thm.~\ref{thm:lam3.4} \\ \hline
\end{tabular}
\end{table*}

\begin{table*}
  \centering
\caption{Possible liftings with $150 \leq k < 200$}
\label{tab:newbig}
\begin{tabular}{|c|c|c|c|c|}  \hline
$v$ & $k$ & $\lambda$ & type  & nonexistence proof \\ \hline
307 & 153 & 76 & Paley &  $n=2$: Thm.~\ref{thm:pott2.5}, $n=19$: Thm.~\ref{thm:lam3.2} \\
307 & 154 & 77 & complement of Paley &  $n=11$: Thm.~\ref{thm:pott2.5}, $n=7$: Thm.~\ref{thm:lam3.4} \\
311 & 155 & 77 & Paley &  $n=7$: Thm.~\ref{thm:pott2.5}, $n=11$: Thm.~\ref{thm:lam3.4} \\
311 & 156 & 78 & complement of Paley &  $n=2$: Thm.~\ref{thm:pott2.5}, $n=3, 13$: Thm.~\ref{thm:lam3.4} \\
781 & 156 & 31 & (4,5) Singer &  Thm.~\ref{thm:pott2.5} \\
323 & 161 & 80 & TPP(17) &  Thm.~\ref{thm:pott2.5} \\
323 & 162 & 81 & complement of TPP(17) & Thm.~\ref{thm:pott2.5}  \\
331 & 165 & 82 & Paley &  Thm.~\ref{thm:pott2.5} \\
331 & 166 & 83 & complement of Paley & Thm.~\ref{thm:lam3.5}  \\
%183 & 169 & 156 & complement of (2,13) Singer &   \\
677 & 169 & 42 & Type B (Hall) &  search \\
347 & 173 & 86 & Paley & $n=2$: Thm.~\ref{thm:pott2.5}, $n=43$: Thm.~\ref{thm:lam3.4}  \\
347 & 174 & 87 & complement of Paley &  Thm.~\ref{thm:pott2.5} \\
359 & 179 & 89 & Paley &  Thm.~\ref{thm:lam3.4} \\
359 & 180 & 90 & complement of Paley & $n=2, 3$: Thm.~\ref{thm:pott2.5}, $n=5$: Thm.~\ref{thm:lam3.4}  \\
367 & 183 & 91 & Paley &  $n=7$: Thm.~\ref{thm:pott2.5}, $n=13$: Thm.~\ref{thm:lam3.4} \\
2380 & 183 & 14 & (3,13) Singer & OPEN  \\
367 & 184 & 92 & complement of Paley & $n=2$: Thm.~\ref{thm:pott2.5}, $n=23$: Thm.~\ref{thm:lam3.4}  \\
379 & 189 & 94 & Paley &  $n=2$: Thm.~\ref{thm:pott2.5}, $n=47$: Thm.~\ref{thm:lam3.4} \\
379 & 190 & 95 & complement of Paley & Thm.~\ref{thm:pott2.5}  \\
383 & 191 & 95 & Paley & Thm.~\ref{thm:lam3.4}  \\
383 & 192 & 96 & complement of Paley & $n=2$: Thm.~\ref{thm:pott2.5}, $n=3$: Thm.~\ref{thm:lam3.4}  \\ \hline
\end{tabular}
\end{table*}

\begin{table*}
  \centering
\caption{Possible liftings with $200 \leq k < 256$}
\label{tab:newbigger}
\begin{tabular}{|c|c|c|c|c|}  \hline
  $v$ & $k$ & $\lambda$ & type  & nonexistence proof \\ \hline
419 & 209 & 104 & Paley & Thm.~\ref{thm:pott2.5} \\
419 & 210 & 105 & complement of Paley & Thm.~\ref{thm:pott2.5} \\
431 & 215 & 107 & Paley & Thm.~\ref{thm:pott2.5} \\
431 & 216 & 108 & complement of Paley & Thm.~\ref{thm:pott2.5} \\
439 & 219 & 109 & Paley & Thm.~\ref{thm:lam3.4} \\
439 & 220 & 110 & complement of Paley &  $n=2$: Thm.~\ref{thm:pott2.5}, $n=5, 11$: Thm.~\ref{thm:lam3.4} \\
443 & 221 & 110 & Paley &  Thm.~\ref{thm:pott2.5} \\
443 & 222 & 111 & complement of Paley & $n=3$: Thm.~\ref{thm:pott2.5}, $n=37$: Thm.~\ref{thm:lam3.4} \\
901 & 225 & 56 & Storer \cite{storer} & search \\
463 & 231 & 115 & Paley &  Thm.~\ref{thm:pott2.5} \\
463 & 232 & 116 & complement of Paley &  Thm.~\ref{thm:pott2.5} \\
467 & 233 & 116 & Paley & $n=2$: Thm.~\ref{thm:pott2.5}, $n=29$: Thm.~\ref{thm:lam3.4} \\
467 & 234 & 117 & complement of Paley & Thm.~\ref{thm:pott2.5} \\
479 & 239 & 119 & Paley & Thm.~\ref{thm:lam3.4} \\
479 & 240 & 120 & complement of Paley & Thm.~\ref{thm:pott2.5} \\
487 & 243 & 121 & Paley & Thm.~\ref{thm:lam3.4} \\
487 & 244 & 122 & complement of Paley & $n=2$: Thm.~\ref{thm:pott2.5}, $n=61$: Thm.~\ref{thm:lam3.4} \\
491 & 245 & 122 & Paley & $n=2$: Thm.~\ref{thm:pott2.5}, $n=61$: Thm.~\ref{thm:lam3.4} \\
491 & 246 & 123 & complement of Paley & $n=3$: Thm.~\ref{thm:lam3.5}, $n=41$: Thm.~\ref{thm:pott2.5} \\
499 & 249 & 124 & Paley & Thm.~\ref{thm:pott2.5} \\
499 & 250 & 125 & complement of Paley & Thm.~\ref{thm:pott2.5} \\
503 & 251 & 125 & Paley & Thm.~\ref{thm:lam3.4} \\
503 & 252 & 126 & complement of Paley & $n=2$: Thm.~\ref{thm:pott2.5}, $n=3, 7$: Thm.~\ref{thm:lam3.4} \\
511 & 255 & 127 & $(8,2)$ Singer & Thm.~\ref{thm:pott2.5} \\
\hline
\end{tabular}
\end{table*}

\begin{table*}
  \centering
  \caption{Lifts for complements of $(d,q)$ Singer difference sets from searches, along with
    the number of inequivalent liftings for each.  The $n$ values are
    the largest values that have a lifting; liftings exist for each
    divisor of $n$ and no other values.}
  \label{tab:singer}
\begin{tabular}{|c|c||c|c|c|c||c|} \hline
$d$ & $q$ & $m$ & $n$ & $k$ & $\lambda$ & \# inequivalent \\ \hline \hline
$2$ & $2$ & $7$ & $2$  &  $4$ &  $1$ & $1$ \\
$2$ & $3$ & $13$  & $2$  &  $9$ &  $3$& $2$ \\
$2$ & $4$ & $21$  & $6$  &  $16$ &  $2$ & $1$ \\
$4$ & $2$ & $31$ & $2$  &  $16$ &  $4$ & $2$ \\
$2$ & $5$ & $31$ & $4$  &  $25$ &  $5$ &  $2$  \\
$3$ & $3$ & $40$ & $2$  &  $27$ &  $9$ & $3$ \\
$2$ & $7$ & $57$ & $6$  &  $49$ &  $7$ & $2$ \\ \hline
$2$ & $8$ & $73$ & $14$  &  $64$ &  $7$ &  $1$\\
$3$ & $4$ & $85$ & $3$  &  $64$ &  $16$ &  $2$\\
$6$ & $2$ & $127$ & $2$  &  $64$ &  $16$ &  $4$ \\
$2$ & $9$ & $91$ & $8$  &  $81$ &  $9$ &  $2$\\
$4$ & $3$ & $121$ & $2$  &  $81$ &  $27$ &  $2$\\
$2$ & $11$ & $133$ & $10$  &  $121$ &  $11$ &  $??$ ($2$ for $n=2$)\\
$3$ & $5$ & $156$ & $4$  &  $125$ &  $25$ &   $??$ ($2$ for $n=2$)\\
$2$ & $13$ & $183$ & $12$  &  $169$ &  $13$ &  $??$\\
$5$ & $3$ & $364$ & $2$  &  $243$ &  $81$ &  $12$\\
$8$ & $2$ & $511$ & $2$  &  $256$ &  $64$ &  $5$\\
$4$ & $4$ & $341$ & $6$  &  $256$ &  $32$ &   $??$ ($2$ for $n=2$, $1$ for $n=3$)\\
$2$ & $16$ & $273$ & $30$  &  $256$ &  $8$ &  $??$ ($1$ for $n=2,3$)\\
\hline
\end{tabular}
\end{table*}

\section{Circulant weighing matrices}\label{sec:cw}

%% A $(v,k,\lambda)$-signed difference set is a group ring element
%% $D=\sum s_i d_i$, where the $s_i \in \{ \pm 1\}$, that satisfies
%% (\ref{eq:gpring}).
%% Signed difference sets are a generalization of difference sets defined
%% in \cite{gordon23}.   A circulant weighing matrix $CW(n,k)$ is equivalent to a
%% signed difference set with $\lambda=0$.  It is well known that a
%% $CW(n,k)$ only exists when $k$ is a square.

As discussed in the introduction, a circulant weighing matrix
$CW(n,k)$ is a cyclically symmetric matrix satisfying (\ref{eq:cw}),
and is equivalent to a group ring element $W$ with coefficients in $\{-1,0,1\}$ 
satisfying (\ref{eq:gpring}) with $\lambda=0$.
It is well known that a
$CW(n,k)$ only exists when $k$ is a square.
If the elements of $W$ with nonzero coefficients are not 
contained in a coset of $\Znum_t$ for any proper divisor $t$ of $n$, 
then $W$ is called {\em proper}.

%it is not a
%multiple of a smaller matrix, i.e. its group ring representation
%$A$ is not equal to $B^{(t)}$ for some divisor $t$ of $n$.

Leung and Schmidt \cite{ls2011} showed that there are only a finite
number of proper $CW(n,k)$ for a given $k$ when $k$ is an odd prime power.
All the proper $CW(n,k)$ are known for $k=4$ (see \cite{eh}),
$9$ (see \cite{aalms08}), and 16 (see \cite{arasu_etal_2006},
\cite{agz}).
A preprint of Leung and Ma \cite{lm2} claimed that the only proper
$CW(n,25)$ have $n=31,33,62,71, 124,142$, but was never published.

There are two main methods for constructing proper CW's.  One is the
Kronecker product  construction of Arasu and Seberry \cite{arasuseberry},
which accounts for almost all of the proper $CW(n,k)$ for $k$
not a prime power, and all the infinite classes except for $CW(2m,2^2)$
\cite{eh} and $CW(48m,6^2)$ \cite{ss13}:

\begin{theorem}\label{thm:kronecker}
  If a proper $CW(n_1,k_1)$ and proper $CW(n_2,k_2)$ exist with $\gcd(n_1,n_2)=1$,
  then they may be used to construct a proper $CW(n_1 n_2,k_1 k_2)$
\end{theorem}

The other construction uses cyclic relative difference sets; see
\cite{arasu_etal_2006}:

\begin{theorem}\label{thm:rds}
  If a cyclic $(m,2n,k,\lambda$)-RDS exists with $m$ and $n$  odd,
  then there is a proper $CW(mn,k)$.
\end{theorem}

%%%%%%%%%%%%%%%%%%%%%%%%%%%%%%%%%%%%%%%%%%%%%%%%%%%%%%%%%%%%
\begin{proof}
  The proof appears in more generality (for divisible difference sets
  in abelian groups) in Ang's thesis \cite{ang}, which is not widely
  available, so we give it here.

  Let $G=\Znum_{2mn}$ and $G'=\Znum_{mn}$.
  We will represent
  $G$ as $G' \times \Znum_2$, and write
  elements of $G$ as $(g,u)$, for $g \in G'$ and $u \in \Znum_2$.
  Similarly, for $N$ the subgroup of $G$ of order $2n$, let $N'$ be
  the subgroup of $G'$
  of order $n$ and write elements of $N$ as $(n,u)$.
  We will write the group ring element for the $(m,2n,k,\lambda)$-RDS as
  $$
  R =  (X,0) + (Y,1) = \sum_{x_i \in X} (x_i,0) + \sum_{y_j \in Y} (y_j,1),
  $$
  for $X, Y \subset G'$.
  Then from 
  (\ref{eq:rds_gpring}), we have
  \begin{eqnarray*}
    R  R^{-1} & = & \left( (X,0) + (Y,1) \right) \left( (X,0) + (Y,1) \right)^{-1} \\
    & = &    (X X^{-1},0) + (Y Y^{-1},0) +  (Y X^{-1},1) + (X Y^{-1},1) \\
    & = & k + \lambda (G - N) = k + \lambda \left( (G' - N',0) +  (G'
    - N',1) \right) .
  \end{eqnarray*}

  Equating terms in cosets of $\Znum_2$, we have
  \begin{equation}\label{eq:rdstocw0}
  X X^{-1} + Y Y^{-1} = k + \lambda(G'-N')    
  \end{equation}
  and
  \begin{equation}\label{eq:rdstocw1}
  Y X^{-1} + X Y^{-1} =  \lambda(G'-N'),
  \end{equation}
  so
  $$
  (X-Y)(X-Y)^{-1} = k.
  $$
  
  From (\ref{eq:rdstocw0}) and (\ref{eq:rdstocw1}), $X$ and $Y$
  generate $G'$, and so $X-Y$ is not contained in a coset of a proper
  subgroup of $G'$.  $X$ and $Y$ are
  disjoint, since if some $x_i = y_j$, then
  $(x_i,0) - (y_j,1) = (0,1) $ is in the forbidden subgroup $N$.
  Therefore $W = X-Y$ is a proper $CW(mn,k)$.
\end{proof}
%%%%%%%%%%%%%%%%%%%%%%%%%%%%%%%%%%%%%%%%%%%%%%%%%%%%%%%%%%%%

This theorem was given as Theorem 6.4 of \cite{agz}, but neglected the
conditions on $m$ and $n$.
Also, Theorem 6.5 of that paper may be
stated in more generality:

\begin{theorem}\label{thm:singercw}
  Let $q$ be a prime power and $d$ odd.  If $n$ is an odd divisor
  of $q-1$, 
  then a proper
  $CW\left(\frac{q^d-1}{q-1} n,q^{d-1}\right)$ exists.
\end{theorem}

\begin{proof}
  From Theorem~\ref{thm:singer}, a cyclic $(m,2n,k,\lambda)$-RDS exists
  for $m=(q^d-1)/(q-1)$, $k=q^{d-1}$, and $n$ any odd divisor of $(q-1)$
  (since $n' =2n$ divides $(q-1)$ if $q$ is is odd, and $2(q-1)$ if
  $q$ is even).
  Since $d$
  is odd, $m$ will be odd, so a proper
  $CW\left(\frac{q^d-1}{q-1} n,q^{d-1}\right)$ exists by Theorem~\ref{thm:rds}.
\end{proof}

Theorem 6.5 gave this result only for $d=3$.  As a result, it missed a
number of CW's, including 
$CW(341,256)$ (from $q=4$, $d=5$; this CW has been constructed using balanced
generalized weighing matrices by Kharagani and Pender \cite{kp}), 
$CW(121,81)$ (from $q=3$, $d=5$) and $CW(127,64)$ (from $q=2$, $d=7$).

Table~\ref{tab:proper} is an updated version of Table 11 of
\cite{agz}.  The differences are:
\begin{itemize}
\item some errors and omissions have been corrected,
\item the more general Theorem~\ref{thm:singercw} is used,
\item results of new exhaustive searches have been included,
\item some parameters have both CW's coming from
  one of the constructions as well as other
  sporadic CW's.  Those are now indicated.
\end{itemize}

Note that all but one of the entries with $k>9^2$ not a prime power come from the RDS or
Kronecker constructions.  This is the point where computer searches
(the source of many of the CW's for smaller values of $k$) become
infeasible.  It seems clear that there are many unknown CW's, which
will require new construction techniques to discover.

%% \begin{itemize}
%% \item there are $CW(114,7^2)$, but not from an RDS
%% \item the existence of $CW(732,13^2)$ and $CW(762,19^2)$ are open.
%% \item $CW(121,9^2)$ may be constructed from a $(121,2,81,27)$-RDS.
%% \item $CW(73,8^2)$ comes from a $(73,2,64,28)$-RDS
%% \item  $CW(31,4^2)$ comes from a $(31,2,16,4)$-RDS
%% \item  $CW(62,5^2)$ does not come from an RDS (splitting rule).
%% \item  $CW(511,16^2)$ comes from both the product construction and a $(511,2,256,64)$-RDS
%% \item  $CW(182,9^2)$ and $CW(364,9^2)$ do not come from an RDS
%%   (splitting rule).  182 does exist. How about 364?
%% \item  $CW(614,17^2)$ does not come from an RDS (splitting rule) Does one exist?
%% \item  $CW(366,13^2)$ does not come from an RDS (splitting rule) Does one exist?
%% \item  $CW(182,3^2)$ was missing from the table (product construction)
%% \item  $CW(312,9^2)$ was missing from the table (product construction)
%% \item  $CW(806,12^2)$ was missing from the table (product construction)
%% \item  $CW(819,12^2)$ was missing from the table (product construction)
%% \item  $CW(858,15^2)$ was missing from the table (product construction)
%% \end{itemize}

\newcommand{\rds}[1]{\underline{#1}}
\newcommand{\othercon}[1]{{\boldmath{#1}}}
\newcommand{\compsearch}[1]{\fbox{#1}}
\newcommand{\prodcon}[1]{{#1}}

\begin{table*}
  \centering
\caption{Known proper $CW(n,k)$ with $n < 1000$ and $k \leq
  19^2$. Most CW's come from Theorem~\ref{thm:kronecker}.
  Underlined numbers have CW's coming from
  Theorem~\ref{thm:rds}, numbers in {\bf bold} come from other
  constructions (\cite{arasu_etal_2006}, \cite{ss13}),
  and  entries in boxes have sporadic $CW$s with constructions only
  for those parameters (generally an older result or computer search).
  Entries $cm$ are for all $m$ such that $cm
  \geq k$.
  %Underlined entries in bold have both CW's from one of the constructions
%  as well as other sporadic CW's.
  $CW(217,8^2)$ is the only entry with matrices from the Kronecker
  construction and additional sporadic ones, and
  $CW(511,16^2)$ is the only entry with matrices from both
the Kronecker and RDS constructions.}
\label{tab:proper}
\begin{tabular}{|c|l|}  \hline
$k$ & Known Proper $CW(n,k)$  \\ \hline
$2^2$	&	\compsearch{$2m$}, \rds{$7$}	\\
$3^2$	&	\rds{$13$}, \compsearch{$24$}, \compsearch{$26$}	\\
$4^2$	&	\prodcon{$14m$}, \rds{$21$}, \rds{$31$}, \othercon{$62$}, \compsearch{\rds{$63$}}	\\
$5^2$	&	\rds{$31$}, \compsearch{$33$}, \compsearch{$62$}, \compsearch{$71$}, \compsearch{$124$}, \compsearch{$142$}	\\
$6^2$	&	\prodcon{$26m$}, \othercon{$48m$}, \prodcon{$91$}, \prodcon{$168$}, \prodcon{$182$} \\
$7^2$	&	\rds{$57$}, \compsearch{$87$},  \compsearch{$114$}, \rds{$171$}	\\
$8^2$	&	\prodcon{$42m$}, \prodcon{$62m$}, \rds{$73$}, \compsearch{\rds{$127$}}, \compsearch{$217$}, \othercon{$254$}, \prodcon{$434$}, \rds{$511$}	\\
$9^2$	&	\rds{$91$}, \compsearch{\rds{$121$}},  \compsearch{$182$}, \prodcon{$312$}\\
$10^2$	&	\prodcon{$62m$}, \prodcon{$66m$}, \prodcon{$142m$}, \prodcon{$217$}, \prodcon{$231$}, \prodcon{$434$}, \prodcon{$497$}, \prodcon{$868$}, \prodcon{$994$} \\
$11^2$	&	\rds{$133$}, \rds{$665$}	\\
$12^2$	&	\prodcon{$182m$}, \prodcon{$273$}, \prodcon{$336m$}, \prodcon{$403$}, \prodcon{$546$}, \prodcon{$744$}, \prodcon{$806$}, \prodcon{$819$}	\\
$13^2$	&	\rds{$183$}, \rds{$549$}	\\
$14^2$	&	\prodcon{$114m$}, \prodcon{$174m$}, \prodcon{$342m$}, \prodcon{$399$}, \prodcon{$609$}, \prodcon{$798$}	\\
$15^2$	&	\prodcon{$403$}, \prodcon{$429$}, \prodcon{$744$}, \prodcon{$806$}, \prodcon{$858$}, \prodcon{$923$} \\
$16^2$	&	\prodcon{$146m$}, \prodcon{$254m$}, \rds{$273$}, \rds{$341$}, \prodcon{$434m$}, \rds{$511$}, \prodcon{$651$}, \othercon{$682$}, \rds{$819$}, \prodcon{$889$}	\\
$17^2$	&	\rds{$307$}	\\
  $18^2$	&	\prodcon{$182m$}, \prodcon{$242m$}, \prodcon{$624m$}
  , \prodcon{$847$}	\\
$19^2$	&	\rds{$381$}	\\ \hline
\end{tabular}
\end{table*}

\vspace{0.1in}

\noindent
    {\bf Data Availability:} All the data generated for this paper is
available at {\tt zenodo}; see links at \cite{ljcr}.

\printbibliography

\end{document}